\documentclass[11pt]{article}
\usepackage{amsmath,amsthm,amssymb,mathrsfs}
\usepackage{bm,pifont,xcolor, tikz}
\usetikzlibrary{calc}
\usetikzlibrary{positioning, shapes, calc}
\makeatletter

\newcommand{\Rmnum}[1]{\expandafter\@slowromancap\romannumeral #1@}
\newcommand{\scalebinom}[2]{\left(\!\!\begin{array}{c}#1\\#2\end{array}\!\!\right)}

\makeatother

\usepackage{enumitem}
\usepackage{geometry}
\usepackage[colorlinks=true,
linkcolor=black,citecolor=black,
urlcolor=black]{hyperref}

\makeatletter
\def\@seccntDot{.}
\def\@seccntformat#1{\csname the#1\endcsname\@seccntDot\hskip 0.5em}
\renewcommand\section{\@startsection{section}{1}{\z@}%
{18\p@ \@plus 6\p@ \@minus 3\p@}%
{9\p@ \@plus 6\p@ \@minus 3\p@}%
{\large\bfseries\boldmath}}
\renewcommand\subsection{\@startsection{subsection}{2}{\z@}%
{12\p@ \@plus 6\p@ \@minus 3\p@}%
{3\p@ \@plus 6\p@ \@minus 3\p@}%
{\bfseries\boldmath}}
\renewcommand\subsubsection{\@startsection{subsubsection}{3}{\z@}%
{12\p@ \@plus 6\p@ \@minus 3\p@}%
{\p@}%
{\bfseries\boldmath}}
\makeatother

\usepackage{microtype}

\theoremstyle{plain}
\newtheorem{theorem}{Theorem}[section]
\newtheorem{lemma}{Lemma}[section]

\newtheorem{proposition}{Proposition}[section]

\theoremstyle{definition}

\newtheorem{remark}{Remark}[section]

\newtheorem{claim}{Claim}[section]

\allowdisplaybreaks
\title{\bf The minimum number of wheel-saturating edges}
\author{Xianglei Chen$^{a}$, Changxiang He$^{a}$\footnote{Corresponding author:cxhe@usst.edu.cn}\; 
\footnote{This research is supported by the National Natural Science Foundation of China (Nos. 12271362, 12671416).}\\
{\small $^{a}\,$School of Mathematics, Faculty of Basic Sciences, University of Shanghai for Science and Technology,} \\ 
{\small Shanghai 200093, China} }

\date{}
\begin{document}
		\maketitle
	\begin{abstract}
	Let $G$ be an $H$-free graph. A non-edge $e$ of $G$ is called an $H$-saturating edge  of $G$ if  $G+e$ contains a copy of $H$. For an integer $p\geq 4$, let $W_p:=K_1\vee C_{p-1}$, and let $F_{p}:=K_1\vee P_{p-1}$.
  Let $f_H(n,m)$ denote the minimum number of $H$-saturating edges of an $n$-vertex $H$-free graph with exactly $m$ edges, and let $\operatorname{ex}(n,H)$  denote the Tur{\'a}n number. In this paper, for each fixed integer $k\geq 3$, we determine the exact value of 
	  $f_{W_{2k+1}}(n,\operatorname{ex}(n,F_{2k})+1)$
and show that $f_{W_{2k}}(n,\operatorname{ex}(n,F_{2k-1})+1)=\left(\frac{2}{33}+o(1)\right)n^2.$
	Moreover, we characterize the corresponding extremal graphs.
	
	\medskip
	
	\begin{flushleft}
	\end{flushleft}
\end{abstract} 
\section{Introduction}

  Extremal graph theory has its origins in the classical work of Mantel \cite{Mantel},
  and Turán-type problem have since become a classic problem in extremal graph theory. Given a graph $H$,  a graph is said to be $H$-{\it free } if it contains no copy of $H$ as a subgraph.  The {\it Turán number} of $H$, denoted by $\operatorname{ex}(n,H)$, is defined as the maximum number of edges in an $n$-vertex $H$-free graph.   Let $G$ be an $H$-free graph and $e$ be a {\it non-edge} of $G$ (i.e., an edge in the complement of $G$). We say that $e$ is an $H$-{\it saturating edge} of $G$ if the graph $G + e$ contains a copy of $H$. Indeed, an $H$-free graph $G$ is maximal if and only if every non-edge of $G$ is an $H$-saturating edge. These two problems are closely interconnected.

For a $H$-free graph $G$, let $f_H(G)$ denote the number of $H$-saturating edges of $G$. we also define $f_H(n, m)$ to be the {\it minimum number of} $H$-{\it saturating edges} of an $n$-vertex $H$-free graph with $m$ edges. Erdős and Tuza \cite{tuza} conjectured that every \(K_4\)-free graph \(G\) with $\lfloor \frac{n^2}{4} \rfloor + 1$  edges admits at least $\Big(1 + o(1)\Big)\frac{n^2}{16}$  \(K_4\)-saturating edges. This conjecture was disproved by Balogh and Liu \cite{2014}.  They constructed an $n$-vertex $K_4$-free graph containing $\lfloor \frac{n^2}{4} \rfloor+1$ edges,    which has only 
$\Big(1+o(1)\Big)\frac{2n^2}{33}$ $K_4$-saturating edges. In addition, they proved   that this construction is the best possible.

\begin{theorem}{\rm(\cite[Theorem 1.2] {2014} ) }
	$f_{K_4}(n,  \operatorname{ex}(n,K_3)+1)=(1+o(1))\frac{2n^2}{33}.$ 
\end{theorem}
In \cite{2014}, Balogh and Liu conjectured that a natural generalization of their construction would remain optimal for $K_p$-free graphs. More precisely, they proposed that $$f_{K_p}\left(n,  \operatorname{ex}(n,K_{p-1})+1\right)=\left(\frac{2(p-2)^2}{p(4p^2-11p+8)}+o(1)\right)n^2,$$ a conjecture that was later confirmed in \cite{2023}.
The case of complete graphs has thus been fully resolved.

In this paper, we establish the result for wheel graphs of all orders.
 For an integer $p\geq 4$, let $$W_p=K_1\vee C_{p-1}.$$
Fix an integer $k\geq 3$. Deleting one rim vertex from $W_{2k+1}$ and $W_{2k}$ gives
$$F_{2k}:=K_1\vee P_{2k-1}, \ \ F_{2k-1}:=K_1\vee P_{2k-2}.$$

Throughout the paper, we let 
\begin{align}\label{l}
    \ell_k=
\begin{cases}
    \left\lfloor\frac{2n+k-2}{4}\right\rfloor, \ \ k \text{ is even},\\
    \left\lfloor\frac{2n+k-1}{4}\right\rfloor, \ \ k \text{ is odd}.
\end{cases}
\end{align}

\begin{theorem}{\rm(\cite[Theorem 3] {tree} ) }
    Let $n$ be a sufficiently large integer. Then, for any $F_{2k}$-free or $F_{2k-1}$-free graph $G$ with $n$ vertices, we have
    \begin{align}\label{1}
        e(G)\leq \max \left\{n_0n_1+\left\lfloor\frac{(k-2)n_0}{2}\right\rfloor: n_0+n_1=n\right\}= \ell_k(n-\ell_k)+\left\lfloor\frac{(k-2)\ell_k}{2}\right\rfloor.
    \end{align} 
\end{theorem}

\begin{remark}
 It is easy to verify directly that
 \begin{align}\label{tunu}
     \text{ex}(n,F_{2k})=\text{ex}(n,F_{2k-1})=\frac{n^2}{4}+\frac{k-2}{4}n+O_k(1).
 \end{align}
 For convenience, we shall use this value whenever necessary in the proof.
\end{remark}

The asymptotic coefficient in Theorem \ref{even} is independent of $k$, whereas the exact value for odd-order wheels depends on $k$, as shown in the following theorem.

\begin{theorem} \label{even}
	Fix $k\geq 3$. Let $n$ be a sufficiently large integer.
    Then, 
    \begin{align}
        f_{W_{2k}}(n,\text{ex}(n,F_{2k-1})+1)=\left(\frac{2}{33}+o(1)\right)n^2.
    \end{align}
\end{theorem}

We next state the result for wheels of odd order.

\begin{theorem}	\label{odd}
	Fix $k\geq 3$. Let $n$ be a sufficiently large integer.
    Then, 
    \begin{align}
        f_{W_{2k+1}}(n,\text{ex}(n,F_{2k})+1)=
   \begin{cases}
    \ell_k-k, \  &(k-2)\ell_k \ \text{is odd},\\
    2(\ell_k-k)-1, \ &(k-2)\ell_k \ \text{is even}.
\end{cases}
    \end{align}
	Equality holds precisely for graphs $$G=H_{\ell_k}\vee \overline{K}_{n-\ell_k},$$
    where $|H_{\ell_k}|=\ell_k$, $H_k$ contains no $P_{2k-1}$, and the degree sequence of $H_{\ell_k}$ is as follows:
    
    (i) $\left(k-1, (k-2)^{\ell_k-1}\right)$ if  $(k-2)\ell_k$ is odd;

    (ii) $\left((k-1)^2, (k-2)^{\ell_k-2}\right)$ if  $(k-2)\ell_k$ is even, with the two vertices of degree $k-1$ nonadjacent.
    
\end{theorem}

The rest of the paper is organized as follows.  
 In Sections 2 and 3, we prove Theorem \ref{even} and Theorem \ref{odd}, respectively.

\section{ Proof of Theorem \ref{even}}

We now turn to $W_{2k}:=K_1\vee C_{2k-1}$. The rim of $W_{2k}$ is odd, and this changes the problem substantially. In contrast with the odd-order case, the sharp construction has $\Theta(n^2)$ $W_{2k}$-saturating  edges, so the linear upper-bound argument used in Section 2 is no longer available.

The proof has two parts. We first construct a five-part blow-up giving the upper bound. This construction is adapted from \cite{2014}. For the other part, we split the lower bound into graphs that are far away from bipartite and graphs that are close to bipartite.

\subsection{Upper bound of Theorem \ref{even}}

The upper bound follows from a slight modification of the construction of Balogh and Liu \cite{2014}. Their construction is a blow-up of a $C_5$ with one chord and blows up its five vertices into independent sets, with asymptotic part ratios
$$|V_1|:|V_2|:|V_3|:|V_4|:|V_5|=4:16:16:15:15,$$
as illustrated in Figure \ref{evengr}.

\begin{figure}[!h]
 	\centering
\begin{tikzpicture}[every node/.style={font=\small}]

 		\node[circle, draw, line width=1.2pt, minimum size=20pt,inner sep=0pt,
    text width=10pt,
    align=center,
    font=\small, label=above:$V_1\ (4)$] (V1) at (0,0) {};
 		\node[circle, draw,  line width=1.2pt, minimum size=20pt,inner sep=0pt,
    text width=10pt,
    align=center,
    font=\small, label=left:$V_2\ (16)$](V2) at (-2,-1.5) {};
 		\node[circle, draw,  line width=1.2pt, minimum size=20pt,inner sep=0pt,
    text width=10pt,
    align=center,
    font=\small,label=right:$V_3\ (16)$](V3) at (2,-1.5) {};
 		\node[circle, draw,  line width=1.2pt, minimum size=20pt,inner sep=0pt,
    text width=10pt,
    align=center,
    font=\small, label=left:$V_4 \ (15)$](V4) at (-1,-3) {};    
 		\node[circle, draw,  line width=1.2pt, minimum size=20pt,inner sep=0pt,
    text width=10pt,
    align=center,
    font=\small, label=right:$V_5\ (15)$](V5) at (1,-3) {};

 		\draw[line width=2pt] (V1) -- (V2);
 		\draw[line width=2pt] (V1) -- (V3);
        \draw[line width=2pt] (V3) -- (V2);
        \draw[line width=2pt] (V3) -- (V5);
        \draw[line width=2pt] (V2) -- (V4);
        \draw[line width=2pt] (V4) -- (V5);
 		

\end{tikzpicture}
 		\caption{A construction for the upper bound of Theorem \ref{even}}
         \label{evengr}
 	
 \end{figure}
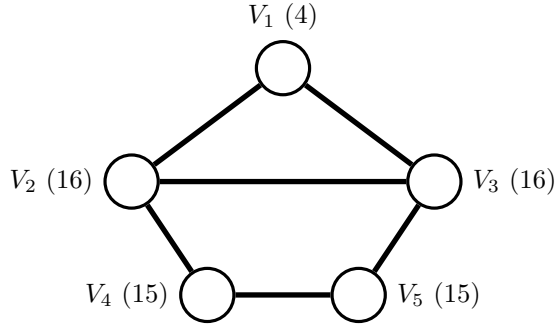
We use the same five-part construction $G'$, modifying only the part sizes by $O_k(1)$ so  that the resulting graph has exactly $\text{ex}(n,F_{2k-1})+1$ edges. And if necessary modifying $O_k(n)$ edges, These modifications affect the number of $W_{2k}$-saturating edges by at most $O_k(n)$. Every vertex neighborhood in this blow-up is bipartite, and hence the graph is $W_{2k}$-free. Moreover, the saturation pattern is the same as in the construction of Balogh and Liu \cite{2014}, and the same counting argument gives $$f_{W_{2k}}\left(G'\right)=\left(\frac{2}{33}+o(1)\right)n^2.$$
Then $f(W_{2k};n,\text{ex}(n,F_{2k-1})+1)\leq \left(\frac{2}{33}+o(1)\right)n^2.$

\subsection{Lower bound of Theorem \ref{even}}

Fix $k\geq 3$. We prove the lower bound by contradiction, first reducing to the near-bipartite case. We then analyze a maximum cut and show that the resulting structure forces sufficiently many $W_{2k}$-saturating edges.

Assume for a contradiction that there are $\delta>0$ and arbitrary large integers $n$ for which an infinite sequence of  $n$-vertex $W_{2k}$-free graphs $G=G_n$ satisfying
\begin{align}\label{contradi}
    e(G)=\text{ex}(n,F_{2k-1})+1,\ \ f_{W_{2k}}(G)\leq \left(\frac{2}{33}-\delta\right)n^2.
\end{align}

We first present the following classical theorem of Erdős and Gallai on paths.

\begin{lemma}{\rm(\cite[Erdős-Gallai Theorem] {expath} ) } \label{expath1}
  Let $t\geq 2$ be an integer. If $G$ is $P_t$-free, then $e(G)\leq \frac{t-2}{2}|G|.$
\end{lemma}

The following claim gives a condition that guarantees a copy of $W_{2k+1}$.

\begin{claim}\label{contain1}
    Let $u\in V(G)$, and let $z_1, z_2, \ldots, z_t \in N(u)$ be distinct. Suppose that 
    $$C\subseteq N(u)\cap \bigcap_{i=1}^{t}N(z_i).$$
    If $G[C]$ contains $t$ vertex-disjoint paths $P_1,P_2,\ldots,P_t$, 
    such that 
    $$\sum_{i=1}^{t}e(P_i)=2k-1-2t,$$
    then $G$ contains a copy of $W_{2k}$.
\end{claim}

\begin{proof}
    Since every vertex of $C$ is adjacent to each $z_i$,
    $$z_1P_1z_2P_2\cdots z_tP_tz_1$$
    is a cycle. Moreover,
    $$t+\sum_{i=1}^{t}|P_i|=2t+\sum_{i-1}^{t}e(P_i)=2k-1.$$
    Thus this cycle is a $C_{2k-1}$ contained in $N(u)$, and together with $u$ it forms a copy of $W_{2k}$.
\end{proof}

 We next reduce the global structure to the one relevant to the result by Balogh and Liu \cite{2014}. For a positive integer $t$, let $K_4(t)$ denote the complete 4-partite graph with $t$ vertices in each part, and let $K_4^{(4)}(t)$ denote the complete 4-partite 4-uniform hypergraph with $t$ vertices in each part, whose hyperedges consist of all 4-sets containing exactly one vertex from each part.

 \begin{claim}\label{bi}
     $G$ is $o(n^2)$-close to bipartite. 
 \end{claim}

\begin{proof}
    We first show that $G$ contains a spanning $K_4$-free graph obtained by deleting $o(n^2)$ edges.
Note that $G$ contains only $o(n^4)$ copies of $K_4$. Suppose, to the contrary, that $G$ contains $\Omega(n^4)$ copies of $K_4$. Regarding these copied as hyperedges of a 4-uniform hypergraph $\mathcal{H}$, we have $e(\mathcal{H})=\Omega(n^4)$. By Erdős's theorem on complete multipartite uniform hypergraphs\cite[Theorem 1]{hyper}, every $K_4^{(4)}(k-1)$-free 4-uniform hypergraph on $n$ vertices has $o(n^4)$ edges. Hence $\mathcal{H}$ contains $K_4^{(4)}(k-1)$, which implies that $G$ contains $K_4(k-1)$, and therefore $W_{2k}$, contradicts that $G$ is $W_{2k}$-free.  

    By Ruzsa-Szemer{\'e}di Removal Lemma \cite{Ruzsa}, there exists a spanning $K_4$-free graph $H\subseteq G$ obtained by deleting $o(n^2)$ edges. 
    Then we have
    $$e(H)=\left(\frac{1}{4}+o(1)\right)n^2.$$
    Moreover, since $H\subseteq G$ and $G$ is $W_{2k}$-free,
    $$f_{W_{2k}}(H)\leq f_{W_{2k}}(G). $$

    We next show that $H$ is $o(n^2)$-close to bipartite. Suppose otherwise. Then there exists a fixed $\epsilon>0$ such that, for arbitrary large $n$,
    $$b(H)\geq \epsilon n^2,$$ where
     $$b(H)=\min_{V(H)=A\cup B}\{e(H[A])+e(H[B])\}.$$

     Claim that $H$ contains at least $\tau n^3$ triangles for some $\tau=\tau(\epsilon)>0$. Indeed, otherwise the Ruzsa-Szemer{\'e}di Removal Lemma \cite{Ruzsa} would allow us to remove $o(n^2)$ edges to obtain a triangle-free graph with $\frac{n^2}{4}-o(n^2)$ edges. By Erdős-Simonovits Stability Theorem \cite{Simo},  $H$ would be $o(n^2)$-close to bipartite, contradicting $b(H)\geq \epsilon n^2$.

     Fix a positive integer $s$, and let $S$ be a uniformly random $s$-element subset of $V(H)$. Write $X=e(H[S])$, and let $N(K_3,H[S])$ denote the number of triangles in $H[S]$. Since $H$ contains at least $\tau n^3$ triangles and $e(H)=\frac{n^2}{4}+o(n^2)$, linearity of expectation gives
     $$\mathbb{E}(X)=\frac{\binom{n-2}{s-2}}{\binom{n}{s}}e(H)\geq \left(\frac{1}{4}+o(1)\right)s^2,$$
     and 
     $$\mathbb{E}(N(K_3,H[S]))\geq\frac{\binom{n-3}{s-3}}{\binom{n}{3}}\tau n^3.$$
     Standard second-moment estimates for uniform vertex sampling yield 
     $$\text{Var}(X)=O(s^3),\ \text{Var}(N(K_3,H[S]))=O(s^5).$$
     Hence, by Chebyshev's inequality,
     $$\text{Pr}\left(N(K_3,H[S])>0, \left|X-\frac{s^2}{4}\right|\leq s^{\frac{7}{4}}\right)\geq 1-O_{\tau}(s^{-\frac{1}{2}})-o_n(1).$$
     Thus, taking $n\rightarrow \infty$ first and then $s\rightarrow \infty$, the random induced subgraph $H[S]$ contains a triangle and has $\left(\frac{1}{4}+o_s(1)\right)s^2$ edges with probability tending to one.

     Since $H[S]$ is $K_4$-free, the stability result of Balogh and Liu \cite{2014} implies that 
     $$f_{K_4}(H[S])\geq \left(\frac{2}{33}-o_s(1)\right)s^2.$$
     Consequently,
      \begin{align}\label{exx}
         \liminf_{n\rightarrow\infty}{\mathbb{E}(f_{K_4}(H[S]))}\geq  \left(\frac{2}{33}-o_s(1)\right)s^2.
     \end{align}

     On the other hand, consider a non-edge $xy$ of $H$ that is not a $W_{2k}$-saturating edge. If $H[N_H(x)\cap N_H(y)]$ contains a path $P_{2k-2}$, then adding $xy$ would create a $W_{2k}$ with hub $x$. Thus $H[N_H(x)\cap N_H(y)]$ is $P_{2k-2}$-free. By Lemma \ref{expath1},
     $$e\left(H[N_H(x)\cap N_H(y)]\right)\leq (k-2)n.$$
     For $xy$ to become a $K_4$-saturating edge in $H[S]$, the set $S$ must contain $x,y$ and both endpoints of some edge in this common-neighborhood graph. Hence, by the union bound, the probability of this event is at most
     $$(k-2)n \frac{\binom{n-4}{s-4}}{\binom{n}{s}}.$$
     Every $W_{2k}$-saturating edge of $H$ contributes at most $ \frac{\binom{n-2}{s-2}}{\binom{n}{s}}$ to the expectation. Summing over all non-edges therefore yields
      \begin{align}\label{exx2}
         \mathbb{E}(f_{K_4}(H[S]))\leq \frac{\binom{n-2}{s-2}}{\binom{n}{s}}f_{W_{2k}}(H)+\binom{n}{2}(k-2)n \frac{\binom{n-4}{s-4}}{\binom{n}{s}}.
     \end{align}

     By initial assumption, 
     $$f_{W_{2k}}(H)\leq f_{W_{2k}}(G)\leq \left(\frac{2}{33}-\delta\right)n^2.$$
    Combining this with (\ref{exx}) and (\ref{exx2}), and, for each fixed $s$, letting $n\rightarrow \infty$, we obtain 
     $$\left(\frac{2}{33}-o_s(1)\right)s^2\leq \left(\frac{2}{33}-\delta\right)s(s-1).$$
    Dividing by $s^2$ and then letting  $s\rightarrow \infty$ gives
     $$\frac{2}{33}\leq\frac{2}{33}-\delta,$$
     a contradiction.
\end{proof}



Let $V(G)=X\cup Y$ be a maximum cut, and define
$$A=e\left(G[X]\right)+e\left(G[Y]\right),\ M=|X||Y|-e(X,Y).$$
By Claim \ref{bi},
$$A=o(n^2),\ \ |X|=\left(\frac{1}{2}+o(1)\right)n,\ |Y|=\left(\frac{1}{2}+o(1)\right)n,\ \ M=o(n^2). $$
For any vertex $v\in V(G)$, let $\psi(v)$ be the number of its missing neighbors across the cut. Then
$$\sum_{v\in V(G)}\psi(v)=2M=o(n^2).$$

We first remove a small set of exceptional vertices to bound the internal degrees while retaining sufficiently many internal edges.

\begin{claim}\label{cc}
    There exists a constant $c>0$ and a set $Z\subseteq V(G)$ with $|Z|=o(n)$ such that, putting $X_0=X\backslash Z$ and $Y_0=Y\backslash Z$,
    $$\Delta(G[X_0])\leq k-2,\ \Delta(G[Y_0])\leq k-2,$$
    and
    $$e\left(G[X_0]\right)+e\left(G[Y_0]\right)\geq \frac{k-2}{4}n+cM+\frac{(|X|-|Y|)^2}{4}-o(n).$$
\end{claim}

\begin{proof}
    For every fixed $\epsilon>0$, we claim that if, for $v\in X$, $|N_X(v)|\geq \epsilon n$, then
    \begin{align}\label{8}
        f_{W_{2k}}(G)\geq \frac{|N_X(v)|^2+|N_Y(v)|^2}{2}-o(n^2).
    \end{align}
    The case $v\in Y$ follows by symmetry.
    Since the cut is maximal,
    $$|N_Y(v)|\geq |N_X(v)|\geq \epsilon n.$$
    As $M=o(n^2)$,
    all but $o(n)$ vertices has $o(n)$ missing cross neighbors. Call these vertices \textit{good}.
    
    Consider any non-edge $uw$ with both endpoints good in $N_X(v)$. Choose distinct good vertices $x_1, x_2,\ldots, x_{k-2}\in N_{X}(v)\backslash\{u,w\}$. Since 
    $$|N(u)\cap N(w) \cap N_{Y}(v)|\geq |N_{Y}(v)|-\psi(u)-\psi(w)=|N_{Y}(v)|-o(n)\geq \epsilon n,$$
    this allows us to choose distinct vertices $y_1, y_2,\ldots, y_{k-1}\in N_{Y}(v)$ such that
    $$uwy_1x_1y_2\cdots x_{k-2}y_{k-1}u$$
    forms a $C_{2k-1}$ in $N_{G+uw}(v)$. Thus $uw$ is a $W_{2k}$-saturating edge. Since all but $o(n)$ vertices of $N_X(v)$ are good and $e(G[X])=o(n^2)$, this gives at least
    \begin{align}\label{9}
        \frac{|N_X(v)|^2}{2}-o(n^2)
    \end{align}
    $W_{2k}$-saturating edges inside $N_X(v)$.

    On the other hand, choose distinct good vertices $a_1, a_2,\ldots, a_{k-1}\in N_{X}(v)$, then
    $$\left|N_Y(v)\bigcap_{i=1}^{k-1}N(a_i)\right|\geq |N_Y(v)|-\sum_{i=1}^{k-1}\psi(a_i)=|N_Y(v)|-o(n)\geq \epsilon n-o(n).$$
    For every non-edge $bb_0$ inside $N_Y(v)\bigcap_{i=1}^{k-1}N(z_i)$, choose distinct vertices $b_1, b_2,\ldots, b_{k-2}\in N_Y(v)\bigcap_{i=1}^{k-1}N(z_i)\backslash \{b,b_0\}$. Then
    $$bb_0a_1b_1a_2\cdots a_{k-2}b_{k-2}a_{k-1}b$$
    forms a $C_{2k-1}$ in $N_{G+bb_0}(v)$. Thus $bb_0$ is a $W_{2k}$-saturating edge.
    Since $e(G[Y])=o(n^2)$, this gives at least 
    $$\frac{|N_Y(v)|^2}{2}-o(n^2)$$
    additional $W_{2k}$-saturating edges inside $N_Y(v)$.
    Combining this with (\ref{9}), we have (\ref{8}).

    Choose a constant $\rho<1$, sufficiently close to 1, such that 
    \begin{align}\label{10}
        \frac{\rho^2}{4(1+\rho)^2}>\frac{2}{33}.
    \end{align}
Claim that every vertex $v\in X$ satisfies
\begin{align}\label{11}
    |N_X(v)|\leq \rho \psi(v)+o(n).
\end{align}
The case $v\in Y$ follows by symmetry.
Suppose otherwise. Then, for some fixed $\epsilon>0$, there exists a vertex $v\in X$ such that
$$|N_X(v)|\geq \rho \psi(v)+\epsilon n.$$
Since the cut is maximal,
$$\rho \psi(v)\leq |N_Y(v)|=\left(\frac{1}{2}+o(1)\right)n-\psi(v).$$
Then
$$0\leq \psi(v)\leq \left(\frac{1}{2(1+\rho)}+o(1)\right)n.$$
 By (\ref{8}),
 $$f_{W_{2k}}(G)\geq \frac{1}{2}\left(\rho^2\psi^2(v)+\left(\frac{n}{2}-\psi(v)\right)^2\right)+o(n^2.)$$
 The function on the right side is decreasing on $0\leq \psi(v)\leq \left(\frac{1}{2(1+\rho)}\right)n,$ and hence
 $$f_{W_{2k}}(G)\geq\left(\frac{\rho^2}{4(1+\rho)^2}+o(1)\right)n^2>\left(\frac{2}{33}+o(1)\right)n^2,$$
 contradicting (\ref{contradi}). Then we have (\ref{11}).

 We next show that if a vertex $u\in X$ has $k-1$ distinct neighbors $z_1, z_2,\ldots, z_{k-1}\in N_{X}(u)$, then
 \begin{align}\label{12}
     \psi(u)+\sum_{i=1}^{k-1}\psi(z_i)\geq \frac{n}{8}.
 \end{align}
 Again, the case $u\in Y$ is symmetric. Suppose otherwise. Then
 $$\left|N_Y(u)\bigcap_{i=1}^{k-1}N(z_i)\right|\geq |Y|-\psi(u)-\sum_{i=1}^{k-1}\psi(z_i)>\left(\frac{3}{8}-o(1)\right)n.$$
 For any non-edge $yy_0$ inside $N_Y(u)\bigcap_{i=1}^{k-1}N(z_i)$, the edge $yy_0$, together with $k-2$ isolated vertices of $\left(N_Y(u)\bigcap_{i=1}^{k-1}N_Y(z_i)\right)\backslash\{y,y_0\}$, forms a linear forest with $k-1$ components and one edge. Since
 $$2k-1-2(k-1)=1,$$
 Claim \ref{contain1}, applied in $G+yy_0$, shows that $yy_0$ is a $W_{2k}$-saturating edge. Consequently,
 $$f_{W_{2k}}(G)\geq\scalebinom{\left|N_Y(v)\bigcap_{i=1}^{k-1}N(z_i)\right|}{2}-e\left(N_Y(u)\bigcap_{i=1}^{k-1}N(z_i)\right)\geq \left(\frac{9}{128}-o(1)\right)n^2>\left(\frac{2}{33}+o(1)\right)n^2,$$
 again contradicting (\ref{contradi}). This proves (\ref{12}). 

 We now remove vertices with relatively large cross defects. By (\ref{11}), there exists a sequence $\theta_n\rightarrow0$ such that, for every $v\in V(G)$,
 $$d_{in}(v)\leq \rho \psi(v)+\theta_n n,$$
 where $d_{in}(v)$ denotes the degree of $v$ within its own part.
 Since $M=o(n^2)$ and $\theta_n=o(1)$, we may choose $\beta_n\rightarrow0$ sufficiently slowly that 
 $$M=o(\beta^2_nn^2),\ \theta_n=o(\beta_n).$$
 Let
 $$Z=\{v\in V(G):\psi(v)\geq \beta_nn\},\ X_0=X\backslash Z,\ Y_0=Y\backslash Z.$$
 Then
 \begin{align}\label{14}
     |Z|\leq \frac{2M}{\beta_nn}=o(n).
 \end{align}
Note that
\begin{align}\label{13}
    \Delta(G[X_0])\leq k-2,\ \Delta(G[Y_0])\leq k-2.
\end{align}
Indeed, every vertex outside $Z$ has fewer than $\beta_nn$ missing cross neighbors. If some vertex $u\in X_0$ has $k-1$ distinct neighbors $z_1, z_2,\ldots, z_{k-1}\in N_{X_0}(u)$, then
 \begin{align}
     \psi(u)+\sum_{i=1}^{k-1}\psi(z_i)\leq (k-1)\beta_nn< \frac{n}{8},
 \end{align}
 for sufficiently large $n$, contradicting (\ref{12}). The same argument applies to $Y_0$.

It remains to estimate the internal edges lost by deleting $Z$. Each missing cross edge contributes at most once to  $\sum_{v\in Z}\psi(v)$, except when both endpoints lie in $Z$. Therefore,
$$\sum_{v\in Z}\psi(v)\leq M+\binom{|Z|}{2}.$$
By (\ref{14}) and the choice of $\beta_n$,
$$|Z|^2=o(M),\ \theta_nn|Z|=o(M)$$
when $M>0$; when $M=0$, the set $Z$ is empty. It follows that
$$\sum_{v\in Z}d_{in}(v)\leq \rho \sum_{v\in Z}\psi(v)+\theta_nn|Z|\leq (\rho+o(1))M. $$
Then
$$e(G[X_0])+e(G[Y_0])\geq A- (\rho+o(1))M.$$
As $$A-M= \frac{k-2}{4}n+\frac{(|X|-|Y|)^2}{4}+O_k(1),$$
$$e(G[X_0])+e(G[Y_0])\geq \frac{k-2}{4}n+ (1-\rho-o(1))M+\frac{(|X|-|Y|)^2}{4}+O_k(1).$$
Since $\rho<1$ is a constant, we obtain, for sufficiently large $n$,
$$(1-\rho-o(1))\geq \frac{1-\rho}{2}:=c.$$
Then, 
$$e(G[X_0])+e(G[Y_0])\geq \frac{k-2}{4}n+ cM+\frac{(|X|-|Y|)^2}{4}-o(n).$$
This complete the proof of Claim \ref{cc}.
\end{proof}

By Claim \ref{cc}, we have $$ |X_0|=\left(\frac{1}{2}+o(1)\right)n,\ |Y_0|=\left(\frac{1}{2}+o(1)\right)n,\ \Delta(G[X_0])\leq k-2,\ \Delta(G[Y_0])\leq k-2.$$
Passing to a subsequence if necessary, let $\kappa_1$ and $\kappa_2$ be the limiting averrage internal degree of $X_0$ and $Y_0$, respectively.
Then
$$0\leq \kappa_1,\kappa_2\leq k-2.$$
Since $e(G[X_0])+e(G[Y_0])\geq \left(\frac{k-2}{4}-o(1)\right)n$ and $e(G[X_0])+e(G[Y_0])= \left(\frac{\kappa_1+\kappa_2}{4}+o(1)\right)n$, we have
\begin{align}\label{14}
    \left(\frac{\kappa_1+\kappa_2}{4}+o(1)\right)n\geq \left(\frac{k-2}{4}-o(1)\right)n.
\end{align}
Dividing by $\frac{n}{4}$ and letting $n\rightarrow\infty$, we obtain
$$\kappa_1+\kappa_2\geq k-2.$$
By symmetry, we may assume $\kappa_1\leq\kappa_2$.

We next use paths in one part to identify a large family of $W_{2k}$-saturating edges in the other.

\begin{claim}\label{15}
    For $1\leq \xi\leq k-2,$ let
    $$r_{\xi}=\left\lceil \frac{2k-3-\xi}{\xi+2}\right\rceil.$$
    If $\kappa_1>\xi-1$, then every non-edge of $G[Y_0]$ with at least one point of degree at least $r_{\xi}$ in $G[Y_0]$ is a $W_{2k}$-saturating edge. The same holds with $X_0$, $\kappa_1$ and $Y_0$, $\kappa_2$ interchanged.
\end{claim}

\begin{proof}
    We prove the first assertion; the second follows by symmetry.
    Suppose $\kappa_1>\xi-1$, and let $xy$ be a non-edge of $G[Y_0]$ such that 
    $$t:=d_{Y_0}(x)\geq r_{\xi}.$$
    By Claim \ref{cc}, $d_{Y_0}(x)\leq k-2 $. Let $z_1,z_2,\ldots, z_t\in N_{Y_0}(x)$, and put
    $$R:=X_0\cap N(x) \cap N(y)\cap \bigcap_{i=1}^{t}N(z_i).$$
    Since every vertex of $Y_0$ has $o(n)$ missing cross neighbors and $t\leq k-2$, we have
    $$\left|R\right|\geq |X_0|-\psi(x)-\psi(y)-\sum_{i=1}^{t}\psi(z_i)=|X_0|-o(n).$$
    Moreover, Since $\Delta(G[X_0])\leq k-2$,
    $$e(G[R])=e(G[X_0])-o(n).$$
    Then $$\frac{2e(G[R])}{|R|}=\kappa_1+o(1)>\xi-1,$$
    for sufficiently large $n$.

    We next show that $G[R]$ contains $t+1$ pairwise vertex-disjoint paths of length $\xi$. By Lemma \ref{expath1}, $G[R]$ contains a path of length $\xi$.
   After removing $j\leq t$ pairwise vertex-disjoint paths of length $\xi$, at most $$j(\xi+1)\leq (k-2)(\xi+1)=O_k(1)$$
   vertices and $O_k(1)$ edges have been deleted. Since $|R|=\left(\frac{1}{2}+o(1)\right)n$, the remaining graph still has average degree
   $$\kappa_1+o(1)>\xi-1.$$
   Applying Lemma \ref{expath1} repeatedly, we obtain $t+1$ pairwise vertex-disjoint paths of length $\xi$ in $G[R]$.
   Since $t\geq r_{\xi}$ and $t\leq k-2$, we have
   $$1\leq 2k-3-2t\leq (t+1)\xi.$$
   Then the $t+1$ disjoint paths can be shortened to form a linear forest with $t+1$ components and exactly $2k-3-2t$ edges. By Claim \ref{contain1}, $G+xy$ contains a copy of $W_{2k}$ centered at $x$.
   Hence $xy$ is a $W_{2k}$-saturating edge, completing the proof.
\end{proof}

We now use Claim \ref{15} to exclude all but two possible configuration of the limiting average internal degrees.

\begin{claim}\label{17}
    Under (\ref{contradi}), either
    $\kappa_1=0$ or $k=4, \ \kappa_1=\kappa_2=1.$
\end{claim}

\begin{proof}
For $Z\in\{X_0,Y_0\}$ and
    $1\leq r_{\xi} \leq k-2$, write 
    $$\tau_{r_{\xi}}(Z):=\left|\{v\in Z:d_{Z}(v)\geq r_{\xi}\}\right|.
    $$
    By Claim \ref{cc}, 
    \begin{align*}
        2e(G[Z])\leq (r_{\xi}-1)(|Z|-\tau_{r_{\xi}}(Z))+(k-2)\tau_{r_{\xi}}(Z).
    \end{align*}
Then
\begin{align}\label{16}
    \tau_{r_{\xi}}(X_0)\geq \max\left\{0,\frac{\kappa_1-r_{\xi}+1}{k-1-r_{\xi}}\right\}\frac{n}{2}-o(n),\
    \tau_{r_{\xi}}(Y_0)\geq \max\left\{0,\frac{\kappa_2-r_{\xi}+1}{k-1-r_{\xi}}\right\}\frac{n}{2}-o(n).
\end{align}
By Claim \ref{cc}, 
$$f_{W_{2k}}(G)\geq \binom{|Z|}{2}-\binom{|Z|-\tau_{r_{\xi}}(Z)}{2}-e(G[Z])=\frac{n}{2}\tau_{r_{\xi}}(Z)-\frac{\tau_{r_{\xi}}^2(Z)}{2}-o(n^2).$$
In particular, if $\tau_{r_{\xi}}(Z)\geq\left(\frac{1}{6}+o(1)\right)n$, then
$$f_{W_{2k}}(G)\geq \left(\frac{5}{72}+o(1)\right)n^2,$$
contradicting (\ref{contradi}).
Similarly, $\tau_{r_{\xi}}(Z)\geq\left(\frac{1}{4}+o(1)\right)n$ yields
$$f_{W_{2k}}(G)\geq \left(\frac{3}{32}+o(1)\right)n^2,$$ again a contradiction.
More generally, if Claim \ref{cc} applies to both parts, possibly with different thresholds, and the corresponding numbers of vertices satisfy 
$$\tau_{r_{\xi}}(X_0)+\tau_{r_{\xi}}(Y_0)\geq \left(\frac{1}{4}+o(1)\right)n,$$
then the same contradiction follows.

Recall that $$0\leq \kappa_1\leq \kappa_2\leq k-2,\ \kappa_1+\kappa_2\geq k-2.$$
Suppose first that $k\geq 6$ and $\kappa_1 >0$. Put
$$\alpha=\min\left\{\kappa_1,\frac{k-2}{2}\right\}, \ \xi=\lceil \alpha \rceil.$$
Then
$$\kappa_1>\xi-1,\ \kappa_2\geq k-2-\alpha.$$
By (\ref{16}),
$$\tau_r(Y_0)\geq \left(\frac{\kappa_2-r_{\xi}+1}{k-1-r_{\xi}}\right)\frac{n}{2}-o(n)\geq \left(1-\frac{\alpha}{k-1-r_{\xi}}\right)\frac{n}{2}-o(n).$$
By the definition of $r_{\xi}$ and the choice of $\alpha$,
we have
$$k-1-r_{\xi}\geq\frac{3}{2}\alpha.$$
Indeed, this follows directly for $\xi\leq 3$, while for $\xi\geq 4$ it follows from $\alpha\leq \xi$.
Then we have
$$\tau_r(Y_0)\geq \left(1-\frac{2}{3}\right)\frac{n}{2}-o(n)=\left(\frac{1}{6}+o(1)\right)n,$$ a contradiction.
Hence $$\kappa_1=0$$ whenever $k\geq 6$.
It remains to consider $k=3,4,5$.

{\bf Case} $k=3$. Suppose $\kappa_1>0$. Let $\xi=1$, then $r_{\xi}=1.$
By Claim \ref{cc}, every non-edge of $G[Y_0]$ incident with a nonisolated vertex is a $W_{2k}$-saturating edge. Since $\kappa_2\geq \frac{1}{2}$,
$$\tau_1(Y_0)\geq \kappa_2\frac{n}{2}-o(n)\geq\left(\frac{1}{4}+o(1)\right)n,$$
contradicting (\ref{contradi}). Hence $\kappa_1=0.$

{\bf Case} $k=4$. Suppose $\kappa_1>0$. Hence $$r_1=2,\ r_2=1.$$
If $0<\kappa_1<1$, then $\kappa_2\geq2-\kappa_1>1$.
Applying Claim \ref{cc} to both parts, we obtain
$$\tau_1(X_0)+\tau_2(Y_0)\geq \frac{\kappa_1n}{4}+\frac{(1-\kappa_1)n}{2}-o(n)\geq \left(\frac{1}{4}+o(1)\right)n,$$
contradicting (\ref{contradi}). 
If $\kappa_1\geq 1$ and $\kappa_2>1$, then
$$\tau_1(X_0)\geq \kappa_1\frac{n}{4}-o(n)\geq\left(\frac{1}{4}+o(1)\right)n,$$
contradicting (\ref{contradi}). Hence $\kappa_1=0,$
again contradicting (\ref{contradi}). Thus the only remaining possibility is $$\kappa_1=\kappa_2=1.$$

{\bf Case} $k=5$. Here $$r_1=2,\ r_2=2.$$
If $0<\kappa_1\leq 1$, then $\kappa_2\geq 2$, and hence
$$\tau_2(Y_0)\geq \frac{(\kappa_2-1)n}{4}-o(n)\geq \left(\frac{1}{4}+o(1)\right)n,$$
contradicting (\ref{contradi}). 
If $\kappa_1>1$, then both parts satisfy Claim \ref{cc} with $\xi=2$.
Thus
$$\tau_2(X_0)+\tau_2(Y_0)\geq \frac{(\kappa_1+\kappa_2-2)n}{4}-o(n)\geq \left(\frac{1}{4}+o(1)\right)n,$$
again contradicting (\ref{contradi}).
Combining these cases with $k\geq 6$, we conclude that either
$\kappa_1=0$ or $k=4,\ \kappa_1=\kappa_2=1$, as claimed. 
\end{proof}

We first establish two structural properties that will be used to exclude the remaining cases in Claim \ref{17}.

\begin{claim}\label{20}
    $M=o(n),\ \Delta(G[X])\leq k-2,\ \Delta(G[Y])\leq k-2.$
\end{claim}

\begin{proof}
By Claim \ref{17}, both remaining configurations satisfy
$$\kappa_1+\kappa_2=k-2.$$
    Since $|X_0|=\left(\frac{1}{4}+o(1)\right)n, |Y_0=\left(\frac{1}{4}+o(1)\right)n|$, we have
    $$e(G[X_0])+e(G[Y_0])=\left(\frac{k-2}{4}+o(1)\right)n.$$
    Comparing this with the second inequality of Claim \ref{cc},
    $$M=o(n),\ (|X|^2-|Y|^2)=o(n).$$
    In particular, every vertex has at most $o(n)$ missing cross neighbors.
    Suppose that some vertex $v$ has $k-1$  distinct internal neighbors $z_1,z_2,\ldots,z_{k-1}$. By the local estimate established in the proof of Claim \ref{cc} and (\ref{12}),
    $$\psi(v)+\sum_{i=1}^{k-1}\psi(z_i)\geq \frac{n}{8}.$$
    However, the left-hand side is at most $kM=o(n)$, a contradiction. Thus,
    $\Delta(G[X])\leq k-2,\ \Delta(G[Y])\leq k-2,$ as claimed.
\end{proof}

It remains to exclude the two configurations in Claim \ref{17}.

{\bf Case 1} $\kappa_1=0$.
Then $\kappa_2=k-2$. Since only $o(n)$ vertices were removed in Claim \ref{cc}, Claim \ref{20} implies 
\begin{align}\label{21}
e(G[X)=o(n),\ e(G[Y])=\frac{k-2}{2}|Y|+o(n).
\end{align}
We first show that $G[X]$ contains an edge. Otherwise, by Claim \ref{20},
$$e(G)\leq |X||Y|+\left\lfloor\frac{k-2}{2}|Y|\right\rfloor\leq \text{ex}(n,F_{2k-1}),$$
contrary to $e(G)=\text{ex}(n,F_{2k-1})+1.$
Fix $xx_0\in E(G[X])$, and let $$T=N(x)\cap N(x_0)\cap Y.$$
By Claim \ref{20}, $|T|=|Y|-o(n)$. Moreover, (\ref{21}) gives
$$ \sum_{v\in Y}(k-2-d_{Y}(v))=o(n).$$
Since every summand is nonnegative integer, all but $o(n)$ vertices of $Y$ have internal degree $k-2$. Furthermore, at most $(k-2)(|Y|-|T|)=o(n)$ vertices of $Y$ have an internal neighbor outside $T$. Consequently, all but $o(n)$ vertices $v\in T$ satisfy
\begin{align}\label{22}
    d_Y(v)=k-2,\ N(v)\cap Y\subseteq T.
\end{align}

Fix such a vertex $v$ and non-edge $vu$ with $u\in T$. Let $z_1,z_2,\ldots,z_{k-2}$ be the internal neighbors of $v$. Then $z_1,z_2,\ldots,z_{k-2}\in T$. Hence their common neighborhood with $v,u$ in $X$ contains the edge $xx_0$. Moreover, as $M=o(n)$ their common neighborhood has at least $$|X|-\psi(v)-\psi(u)-\sum_{i=1}^{k-2}\psi(z_i)\geq|X|-kM=|X|-o(n)$$ vertices.
We may therefore choose a linear forest consisting of $xx_0$ and $k-2$ additional isolated vertices. This forest has $k-1$ components and exactly $$2k-1-2(k-1)=1$$
edge. Applying Claim \ref{contain1} to $G+vu$, $vx$ is a $W_{2k}$-saturating edge.
Since (\ref{22}) holds for all but $o(n)$ vertices of $Y$, and $e(G[Y])=O_k(n)$, all but $o(n^2)$ unordered pairs in $Y$ are $W_{2k}$-saturating edges. Therefore,
$$f_{W_{2k}}(G)\geq\binom{|Y|}{2}-o(n^2)=\left(\frac{1}{8}-o(1)\right)n^2,$$
contradicting (\ref{contradi}).

{\bf Case 2} $k=4,\ \kappa_1=\kappa_2=1$.
By Claim \ref{20}, $$\Delta(G[X])\leq 2,\ \Delta(G[Y])\leq 2,$$
and 
\begin{align}\label{24}
    e(G[X])=\frac{|X|}{2}+o(n),\ e(G[Y])=\frac{|Y|}{2}+o(n).
\end{align}
Note that at least one part contains a vertex of internal degree two. Otherwise, both parts have maximum internal degree at most one, giving 
$$e(G)\leq |X||Y|+\left\lfloor\frac{|X|}{2}\right\rfloor+\left\lfloor\frac{|Y|}{2}\right\rfloor\leq |X||Y|+\max\{|X|,|Y|\}\leq \text{ex}(n,F_8),$$
a contradiction.
By symmetry, choose $w\in X$ with two internal neighbors $w_1,w_2$, and put
$$Q=N(w)\cap N(w_1)\cap N(w_2)\cap Y.$$
Since $M=o(n)$, we have $|Q|=|Y|-o(n).$
We first show that every non-edge of $G[Y]$ incident with a vertex of internal degree two is a $W_{2k}$-saturating edge. Indeed, let $vu$ be such a non-edge, with $d_Y(v)=2$, and let $v_1,v_2$ be the internal neighbors of $v$. The common neighborhood of $v,u,v_1,v_2$ in $X$ has $|X|-o(n)$ vertices. By (\ref{24}) and $\Delta(G[X])\leq 2$, it contains an edge. Together with two additional isolated vertices, this edge forms a linear forest with three components and one edge. By Claim \ref{contain1} implies that $vu$ is a $W_{2k}$-saturating edge.

Next we show that all but $o(n^2)$ non-edges joining two isolated vertices of $Y$ are $W_{2k}$-saturating edges. Since $|Y\backslash Q|=o(n)$ and $\Delta(G[Y])\leq 2$, all but $o(n)$ nonisolated vertices $y\in Y$ lie in $Q$ and have an internal neighbor $y'\in Q$. 
Since $\Delta(G[Y])\leq 2$, there are only $O(n)$ pairs $x,y$, for which $xy\in E(G)$ or the chosen edges $xx'$ and $yy'$ share a vertex. 
For every other pair, $x'xyy'$ is a path of length three in $G+xy$.
In $G+xy$, the path $$x'xyy'$$
lies entirely in $Q$. Together wit one additional isolated vertex of $Q$, it forms a linear forest with two components and three edges. By Claim \ref{contain1}, $xy$ is a $W_{2k}$-saturating edge.
Thus, apart from $o(n^2)$ pairs, only vertex pairs of internal degree types (0,0) and (0,1) may fail to be saturating.

Let $p$ be the number of isolated vertices in $G[Y]$. Since $\Delta(G[Y])\leq 2$, and $2e(G[Y])=|Y|+o(n)$, the number of vertices of internal degree two is  $t+o(n)$.
Then
$$f_{W_{2k}}(G)\geq \binom{|Y|-p}{2}+p(p+o(n))-o(n^2)=\frac{|Y|^2}{3}+\frac{3}{2}\left(p-\frac{|Y|}{3}\right)^2-o(n^2)\geq \left(\frac{1}{12}-o(1)\right)n^2,$$
contradicting (\ref{contradi}). 
Both cases lead to contradictions. This completes the proof of the lower bound in Theorem \ref{even}.
  
 \section{ Proof of Theorem \ref{odd}}

In this section, we prove Theorem \ref{odd} by establishing the upper and lower bounds separately.

\subsection{Upper bound of Theorem \ref{odd}}

We establish the upper bound in Theorem \ref{odd} by providing an explicit construction.
Before presenting the specific construction, we first introduce two useful lemmas.

\begin{lemma}\label{compo}
    For every fixed integer $k\geq 3$ and all sufficiently large integer $N$ with $(k-2)N$ even, there exists a $(k-2)$-regular $N$-vertex graph whose components each have fewer than $2k-1$ vertices.
\end{lemma}

\begin{proof}[Proof of Lemma \ref{compo}]
    Let $r=k-2$.
    Every sufficiently large integer $N$ can be written as 
    $$N=p(r+1)+q(2r),$$
    with $p,q\geq 0$.
    If $r$ is even, $\text{gcd}(r+1, 2r)=1$; If $r$ is odd, $\text{gcd}\left((r+1)/2,r\right)=1$.
    
    Take the disjoint union of $p$ copies of $K_{r+1}$ and $q$ copies of $K_{r,r}$. Each component is $(k-2)$-regular and has at most $2k-4$ vertices.
\end{proof}

\begin{lemma}\label{sat1}
    Let $G=H_q\vee \overline{K}_p$, where $H_q$ is an $q$-vertex $P_{2k-1}$-free graph, $\Delta\left(H_q\right)\leq k-1$, and $p\geq k$. Then $G$ is $W_{2k+1}$-free. Moreover, the $W_{2k+1}$-saturating edges are exactly those lying inside $H_q$ with at least one endpoint of degree $k-1$ in $H_q$.
\end{lemma}

\begin{proof}[Proof of Lemma \ref{sat1}]
    Suppose that $G$ contains a copy of $W_{2k+1}$. If its center $w$ lies in $\overline{K}_p$, then $H_q$ contains a copy of $C_{2k}$, contradicting the assumption that $H_q$ is $P_{2k-1}$-free. Hence $w \in V(H_q)$.
    Let $s$ and $r$ denote the numbers of vertices on the rim that lie in $H_q$ and in $\overline{K}_p$, respectively. Since $\overline{K}_p$ is independent, no two of its vertices can be consecutive on the rim, and therefore $r\leq s$.
    Moreover, $s\leq d_{H_q}(w)\leq \Delta(H_q)\leq k-1$. Consequently, the rim contains at most $2k-2$ vertices, a contradiction. Then $G$ is $W_{2k+1}$-free.

    We next prove the "Moreover" part.
    Suppose that $uv\notin E(H_q)$ and $$d_{H_q}(u)=k-1.$$ In the graph $G+uv$, the vertex $u$ has $k$ neighbors in $H_q$, say $z_1,z_2, \cdots, z_k$, where $z_k=v$. Since $p\geq k$, choose distinct vertices $y_1,y_2, \cdots, y_k\in V\left(\overline{K}_p\right)$. As $G=H_q\vee \overline{K}_p$, $$z_1y_1z_2y_2\cdots z_ky_kz_1$$
    forms a copy of $C_{2k}$ in $N_{G+uv}(u)$. Thus, $uv$ is a $W_{2k+1}$-saturating edge.

Conversely, let $uv\notin E(H_q)$ with $$d_{H_q}(u),d_{H_q}(u)\leq k-2.$$ Then $\Delta\left(H_q+uv\right)\leq k-1$.   
Suppose that $G+uv$ contains a copy of $W_{2k+1}$ with center $w$. If $w\in V(H_q)$, since $\Delta\left(H_q+uv\right)\leq k-1$, the same argument as the beginning shows that the rim contains at most $2k-2$ vertices, a contradiction. 
If $w\in V\left(\overline{K}_p\right)$, then the entire rim lies in $H_q+uv$. Since the new wheel must contain $uv$, deleting $uv$ from the rim yields a copy of $P_{2k}$ in $H_q$, and hence a copy of $P_{2k-1}$, contradicting the assumption that $H_q$ is $P_{2k-1}$-free. Then $uv$ is not a $W_{2k+1}$-saturating edge.

Finally, let $u,v\in V\left(\overline{K}_p\right)$ and suppose that $G+uv$ contains a copy of $W_{2k+1}$ with center $w$. Then $w\notin V\left(\overline{K}_p\right)\backslash\{u,v\}$. If $w$ is $u$ or $v$, then deleting the other endpoint from the rim yields a copy of $P_{2k-1}$ in $H_q$, again a contradiction. If $w\in H_q$, then the rim contains at most $k-1$ vertices of $H_q$, since $\Delta(H_q)\leq k-1$. Hence, it must contain at least $k+1>2$ vertices of $\overline{K}_p$, contradicting the fact that $\overline{K}_p$ is independent. Therefore, no non-edge inside $\overline{K}_p$ is $W_{2k+1}$-saturating edge.
\end{proof}

We now present the construction. Fix $k\geq 3$.
Let $n$ be a sufficiently large integer, let $\ell_k$ be given by (\ref{l}), and let $$G^*=H_{\ell_k}\vee \overline{K}_{n-\ell_k}, $$
    where $H_{\ell_k}$ is $\ell_k$-vertex, $P_{2k-1}$-free graph, and the degree sequence of $H_{\ell_k}$ is as follows:
    
    (i) $\left(k-1, (k-2)^{\ell_k-1}\right)$ if  $(k-2)\ell_k$ is odd;

    (ii) $\left((k-1)^2, (k-2)^{\ell_k-2}\right)$ if  $(k-2)\ell_k$ is even, with the two vertices of degree $k-1$ nonadjacent.
The existence and explicit construction of such a graph $G^*$, as well as the verification that the resulting graph $G^*$ attains the desired extremal bound, are established in the following proposition. 


\begin{proposition}\label{exist}
    The construction of $G^*$ described above is feasible and $G^*$ attains equality in Theorem \ref{odd}.
\end{proposition}
    
\begin{proof}[Proof of proposition \ref{exist}]
We first verify that $G^*$ satisfies the required edge condition.
Since $G^*=H_{\ell_k}\vee \overline{K}_{n-\ell_k}, $ 
$e(G^*)=\ell_k(n-\ell_k)+e(H_{\ell_k}).$ Thus, to obtain $e(G^*)=\ell_k(n-\ell_k)+\left\lfloor\frac{(k-2)\ell_k}{2}\right\rfloor+1$, it suffices to construct $H_{\ell_k}$ such that
$$e(H_{\ell_k})=\left\lfloor\frac{(k-2)\ell_k}{2}\right\rfloor+1.$$
   
    If $(k-2)\ell_k$ is odd, then $k-2$ and $\ell_k$ are both odd. Moreover, $\ell_k-k$ is even.
    Take a copy of $K_k$ and delete a matching covering $k-1$ vertices. The resulting graph has one vertex of degree $k-1$ and all other vertices of degree $k-2$. Combining the result that $\ell_k-k$ is even with Lemma \ref{compo}, the remaining $\ell_k-k$ vertices can be partitioned into $(k-2)$-regular components. Hence $H_{\ell_k}$ has one vertex of degree $k-1$ and all others of degree $k-2$, so
    $$e(H_{\ell_k})=\frac{(k-2)\ell_k+1}{2}=\left\lfloor\frac{(k-2)\ell_k}{2}\right\rfloor+1.$$

    If $(k-2)\ell_k$ is even, we consider two cases. For $k\geq 4$, take $\overline{K_2\cup C_{k-1}}$, which has two nonadjacent vertices of degree $k-1$, while all remaining vertices  have degree $k-2$. By Lemma \ref{compo}, the remaining $\ell_k-k-1$ vertices can be partitioned into $(k-2)$-regular components. For $k=3$, 
    take two disjoint copies of $P_3$, yielding two vertices of degree 2 and all other vertices of degree 1; the remaining vertices are completed by a matching.
    Hence $H_{\ell_k}$ has two vertices of degree $k-1$ and all others of degree $k-2$, so
    $$e(H_{\ell_k})=\frac{(k-2)\ell_k+2}{2}=\left\lfloor\frac{(k-2)\ell_k}{2}\right\rfloor+1.$$

    Next, we count the number of $W_{2k+1}$-saturating edges in $G^*$. By Lemma \ref{sat1}, $G^*$ is $W_{2k+1}$-free. When $(k-2)\ell_k$ is odd, there is exactly one $(k-1)$-degree vertex, $$f_{W_{2k+1}}(G^*)=\ell_k-k.$$ 
    When $(k-2)\ell_k$ is even, there are exactly two $(k-1)$-degree vertices, $$f_{W_{2k+1}}(G^*)=2(\ell_k-k)-1.$$
    This completes the proof.
\end{proof}

For illustration, Figure \ref{Sgraph} shows two extremal graphs corresponding to the two cases above, with fixed choices of parameters.

\begin{figure}[!h]
 	\centering
\begin{tikzpicture}
\tikzstyle{every node}=[circle, fill, inner sep=0pt, minimum width=4pt]

\begin{scope}[xshift=0cm]
 		\node[label=left:$u_1$] (L1) at (0,2) {};
 		\node [label=left:$u_2$](L2) at (0,1) {};
 		\node [label=left:$u_3$](L3) at (0,0) {};
 		\node [label=left:$u_4$](L4) at (0,-1) {};    
 		\node [label=left:$u_5$](L5) at (0,-2) {};

 		\node [label=right:$v_1$](R1) at (3,2) {}; 
        \node [label=right:$v_2$](R2) at (3,1) {}; 
        \node [label=right:$v_2$](R3) at (3,0) {}; 
        \node [label=right:$v_4$](R4) at (3,-1) {}; 
        \node [label=right:$v_5$](R5) at (3,-2) {}; 
 		
 		\draw (L1) -- (R1);
 		\draw (L1) -- (R2);
 		\draw (L1) -- (R3);
 		\draw (L1) -- (R4);
        \draw (L1) -- (R5);
        \draw (L1) -- (L2);
        \draw (L2) -- (L3);
 		\draw (L2) -- (R1);
 		\draw (L2) -- (R2);
 		\draw (L2) -- (R3);
 		\draw (L2) -- (R4);
        \draw (L2) -- (R5);
        \draw (L3) -- (R1);
 		\draw (L3) -- (R2);
 		\draw (L3) -- (R3);
 		\draw (L3) -- (R4);
        \draw (L3) -- (R5);
        \draw (L4) -- (R1);
 		\draw (L4) -- (R2);
 		\draw (L4) -- (R3);
 		\draw (L4) -- (R4);
        \draw (L4) -- (R5);
        \draw (L5) -- (R1);
 		\draw (L5) -- (R2);
 		\draw (L5) -- (R3);
 		\draw (L5) -- (R4);
        \draw (L5) -- (R5);
        \draw (L5) -- (L4);

 		\draw[dashed, red] (L2) to[bend right=55] (L4); 
 		\draw[dashed, red] (L2) to[bend right=55] (L5);

 		\draw[thick] (0,0) ellipse (1cm and 2.5cm);
 		\draw[thick] (3,0) ellipse (1cm and 2.5cm);

        \node[draw=none,fill=none,rectangle] at (1.5, -3) {$n=10,\ \ell_k=5, \ f_{W_7}(G^*)=2$};
 \end{scope}

\begin{scope}[xshift=7.5cm]
 		\node[label=left:$u_1$] (L1) at (0,2) {};
 		\node [label=left:$u_2$](L2) at (0,1.2) {};
 		\node [label=left:$u_3$](L3) at (0,0.4) {};
 		\node [label=left:$u_4$](L4) at (0,-0.4) {};    
 		\node [label=left:$u_5$](L5) at (0,-1.2) {}; 
        \node [label=left:$u_6$](L6) at (0,-2) {};
 		
 		\node [label=right:$v_1$](R1) at (3,2) {}; 
        \node [label=right:$v_2$](R2) at (3,1) {}; 
        \node [label=right:$v_2$](R3) at (3,0) {}; 
        \node [label=right:$v_4$](R4) at (3,-1) {}; 
        \node [label=right:$v_5$](R5) at (3,-2) {}; 
 		
 		\draw (L1) -- (R1);
 		\draw (L1) -- (R2);
 		\draw (L1) -- (R3);
 		\draw (L1) -- (R4);
        \draw (L1) -- (R5);
        \draw (L1) -- (L2);
        \draw (L2) -- (L3);
 		\draw (L2) -- (R1);
 		\draw (L2) -- (R2);
 		\draw (L2) -- (R3);
 		\draw (L2) -- (R4);
        \draw (L2) -- (R5);
        \draw (L3) -- (R1);
 		\draw (L3) -- (R2);
 		\draw (L3) -- (R3);
 		\draw (L3) -- (R4);
        \draw (L3) -- (R5);
        \draw (L4) -- (R1);
 		\draw (L4) -- (R2);
 		\draw (L4) -- (R3);
 		\draw (L4) -- (R4);
        \draw (L4) -- (R5);
        \draw (L5) -- (R1);
 		\draw (L5) -- (R2);
 		\draw (L5) -- (R3);
 		\draw (L5) -- (R4);
        \draw (L5) -- (R5);
        \draw (L5) -- (L4);
        \draw (L6) -- (R1);
 		\draw (L6) -- (R2);
 		\draw (L6) -- (R3);
 		\draw (L6) -- (R4);
        \draw (L6) -- (R5);
        \draw (L5) -- (L6);

 		\draw[dashed, red] (L2) to[bend right=60] (L4); 
 		\draw[dashed, red] (L2) to[bend right=55] (L5);
        \draw[dashed, red] (L2) to[bend right=60] (L6); 
        \draw[dashed, red] (L5) to[bend left=55] (L3); 
        \draw[dashed, red] (L5) to[bend left=55] (L1); 

 		\draw[thick] (0,0) ellipse (1cm and 2.5cm);
 		\draw[thick] (3,0) ellipse (1cm and 2.5cm);

        \node[draw=none,fill=none,rectangle] at (1.5, -3) {$n=11,\ \ell_k=6\ f_{W_7}(G^*)=5$};
 		
\end{scope}
 
\end{tikzpicture}
 		\caption{Examples of extremal graphs for odd wheel (red dashed edges are $W_7$-saturating edges)}
         \label{Sgraph}
 	
 \end{figure}
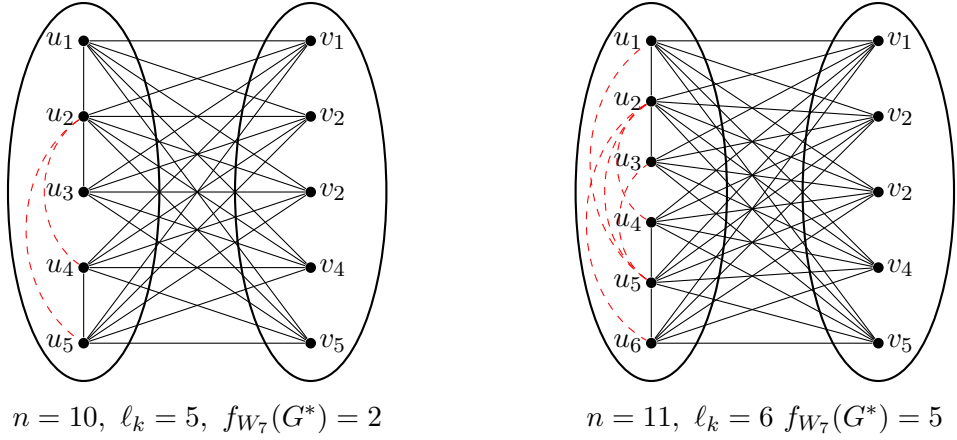

\subsection{Lower bound of Theorem \ref{odd}}

Let $\ell_k$ be given by (\ref{l}). In this section, we reduce every extremal graph to $H_{\ell_k}\vee \overline{K}_{n-\ell_k}$ by analyzing a maximum cut and progressively eliminating the missing cross edges and internal edges on one side.

\subsubsection{Some useful lemmas}

In this section, some useful lemmas needed in the subsequent proofs are presented. We first recall the following classical theorem of Erdős and Gallai on paths.

\begin{lemma}{\rm(\cite[Erdős-Gallai Theorem] {expath} ) } \label{expath}
  Let $t\geq 2$ be an integer. If $G$ is $P_t$-free, then $e(G)\leq \frac{t-2}{2}|G|.$
\end{lemma}

Using Lemma \ref{expath}, we obtain the following result.

\begin{lemma}\label{path}
    Let $G$ be a graph, and let $H\subseteq G$. If $G-H$ contains no two vertex-disjoint copies of $P_{t}$, where $t\geq 2$, then
    $$e(G)\leq \frac{t-2}{2}|G|+\Delta(G)(|H|+t).$$
    For $t=1$, $G-H$ contains two disjoint $P_1$'s whenever $|G|-|H|\geq 2$.
\end{lemma}

\begin{proof}
    If $G-H$ is $P_t$-free, then the result follows immediately from Lemma \ref{expath}. Otherwise, let $P$ be a copy of $P_t$ in $G-H$. Since $G-H-V(P)$ is $P_t$-free, the lemma \ref{expath} gives the desired bound after adding back the edges incident with $H\cup V(P)$, of which there are at most $\Delta(G)(|H|+t)$.
\end{proof}

The following lemma gives a condition that guarantees a copy of $W_{2k+1}$.

\begin{lemma}\label{contain}
    Let $V(G)=V\cup U$, with all edges between $U$ and $V$ present. Let $u\in U$ have distinct neighbors $u_1,u_2,\ldots,u_t\in U$. If $V$ contains a linear forest $F_1\cup F_2\cup \cdots \cup F_t$ satisfying
    $$\sum_{i=1}^{t}e(F_i)=2(k-t),$$
    then $N_G(u)$ contains a $C_{2k}$, and hence $G$ contains a copy of $W_{2k+1}$.
\end{lemma} 

\begin{proof}
    Since all cross edges are present, $$u_1F_1u_2F_2\cdots u_tF_tu_1$$
    is a cycle. Moreover,
    $$t+\sum_{i=1}^{t}|F_i|=2t+\sum_{i=1}^{t}e(F_i)=2k.$$
    Thus this cycle is a $C_{2k}$ contained in $N(u)$, completing the proof.
\end{proof}

The following lemma provides a useful tool that will be invoked several times in the proof.
\begin{lemma}\label{use}
    Let $U$, $V$ be disjoint vertex sets in a $W_{2k+1}$-free graph $G$, and suppose every vertex of $U$ has at most $\gamma$ nonneighbors in $V$. Let $1\leq \zeta\leq k$. Assume that after deleting any at most $(k+1)\gamma$ vertices from $V$, the remaining graph contains a linear forest with $\zeta$ components and $2(k-\zeta)$ edges. Then $$\Delta(G[U])\leq \zeta-1.$$
    Moreover, every same-side non-edge incident with a $(\zeta-1)$-degree vertex of $U$ is a $W_{2k+1}$-saturating edge.
\end{lemma}

\begin{proof}
    We first proof that $\Delta(G[U])\leq \zeta-1$. Otherwise, suppose that $u\in U$ has $\zeta$ neighbors $u_1,u_2,\ldots,u_{\zeta} \in U$. Let $Z\subseteq V$ consist of all vertices that are nonadjacent to at least one vertex of $\{u_0,u_1,\ldots,u_{\zeta}\}$. Since each of these $\zeta+1\leq k+1$ vertices has at most $\gamma$ missing cross neighbors, we have $$|Z|\leq (k+1)\gamma.$$
   Therefore, combining the assumption with Lemma \ref{contain} yields a copy of $W_{2k+1}$, a contradiction. Moreover, this also shows that  every same-side non-edge incident with a vertex of degree $\zeta-1$ in $G[U]$ is a $W_{2k+1}$-saturating edge. 
\end{proof}

\subsubsection{Lower-bound argument of Theorem \ref{odd}}

Fix $k\geq 3$. Let $G$ be a $n$-vertex $W_{2k+1}$-free graph with $\text{ex}(n,F_{2k})+1$ edges and the minimum number of $W_{2k+1}$-saturating edges. By Proposition \ref{exist}, 
\begin{align}\label{upp}
    f_{W_{2k+1}}(G)\leq n+O_{k}(1).
\end{align}
Let $V(G)=X\cup Y$ be a maximum cut of $G$, with $|X|\geq|Y|$. Define
$$A=e\left(G[X]\right)+e\left(G[Y]\right),\ M=|X||Y|-e(X,Y).$$
By the Erdős-Simonovits stability theorem \cite{Simo},
$$A=o(n^2),\ \ |X|=\left(\frac{1}{2}+o(1)\right)n,\ |Y|=\left(\frac{1}{2}+o(1)\right)n,\ \ M=o(n^2). $$

We first prove that $$M=O_{k}(1).$$ The following claim relates the internal degree of each vertex to the number of missing cross edges.

\begin{claim}\label{2}
    $\Delta\left(G[X]\right),\Delta\left(G[Y]\right)\leq k-1+\sqrt{kM}.$
\end{claim}

\begin{proof}
    We prove the case $\Delta\left(G[X]\right)\leq k-1+\sqrt{kM}$; the case $\Delta\left(G[Y]\right)\leq k-1\sqrt{kM}$ follows similarly. Take any vertex $v\in X$. If $d_{G[X]}(v)< k$, the conclusion is trivial. Hence, we assume that $d_{G[X]}(v)\geq k$.
    Since $X\cup Y$ is a maximum cut, $d_{G[Y]}(v)\geq d_{G[X]}(v)$. 

    By double counting, consider the missing cross edges between $N_X(v)$ and $N_Y(v)$ over all $k$-subsets of $N_X(v)$.
    For each $k$-set $S\subseteq N_X(v)$, there are at most $k-1$ vertices of $N_Y(v)$ are adjacent to every vertex of $S$ since $G$ is $W_{2k+2}$-free.
    Thus at least $d_{G[X]}(v)-k+1$ vertices of $N_Y(v)$ miss an edge to every chosen $S$. On the other hand, each missing cross edge is counted at most $\scalebinom{d_{G[X]}(v)-1}{k-1}$ times. Therefore,
    \begin{align*}
        \scalebinom{d_{G[X]}(v)}{k}(d-k+1)\leq M\scalebinom{d_{G[X]}(v)-1}{k-1},
    \end{align*}
and thus
    $$ d_{G[X]}(v)\left(d_{G[X]}(v)-k+1\right)\leq kM.$$
    Since $d_{G[X]}(v)\geq d_{G[X]}(v)-k+1,$ it follows that $$d_{G[X]}(v)\leq k-1+\sqrt{kM}.$$
    Then, as the arbitrary of $v$, $\Delta\left(G[X]\right)\leq k-1+\sqrt{kM}.$
\end{proof}

The following result shows that $M$ is bounded by a constant depending only on $k$.

\begin{claim}\label{MM}
    $M=O_k(1)$.
\end{claim}

\begin{proof}
    We first show that $M=O_k(n)$. Delete every vertex $v\in X$ with
    $$\mu(v):=|Y\backslash N_Y(v)|\geq \frac{|Y|-k+1}{k+1},$$
    and do the same in $Y$. Since $|X|,|Y|=\left(\frac{1}{2}+o(1)\right)n$, only $O_k\left(\frac{M}{n}\right)$ vertices are deleted from each side. The remaining graph on either side has maximum degree at most $k-1$. Indeed, if $v\in X$ has $k$-same-side neighbors $v_1,v_2,\ldots,v_k$, then, by the definition of the deleted vertices, 
    $$\left|\bigcap_{i=0}^{k}N_Y(v_i)\right|\geq |Y|-\sum_{i=0}^{k}\left|Y\backslash N_Y(v_i)\right|>|Y|-(k+1)\frac{|Y|-k+1}{k+1}=k-1,$$
    where $v_0=v$.
    Since the number of common neighbors is an integer, it must be at least $k$.
    Thus $N_G(v)$ contains a $C_{2k}$, a contradiction.

    Together with Claim \ref{2},
    $$A\leq \frac{k-1}{2}n+O_k\left(\frac{M}{n}\left(\sqrt{M}+1\right)\right).$$
    Since $M=|X||Y|+A-e(G)$, $xy\leq\frac{n^2}{4}$, and $e(G)=\frac{n^2}{4}+\frac{k-2}{4}n+O_k(1),$ we obtain $$M\leq \frac{k-1}{4}n+o(M)+O_k(1).$$
    If $M=O_k(1)$, the conclusion follows immediately. Hence, we may assume that $M$ is sufficiently large, so that $$M\leq \frac{k-1}{4}n+o(M),$$
    which implies $M=O_k(n)$. Moreover, for sufficiently large $n$,
    \begin{align}\label{M}
        M\leq \frac{k}{2}n.
    \end{align}

    It remains to prove that $M=O_k(1)$. Suppose, to the contrary, that $M$ is larger than any constant depending only on $k$. Fix a sufficiently small constant $\tau=O_k(1)>0$, and delete from each side every vertex with more than $\tau M$ missing cross neighbors. Since the total missing cross edges on either side is $M$, at most $\frac{1}{\tau}$ vertices are deleted on each side. Delete a few additional vertices if necessary so that the same number $r=O_k(1)$ is removed from each side.

    Let $H$ be the remaining graph, with parts $U$ and $V$, and put $$h=|U|+|V|=n-2r.$$ Every vertex of $H$ has at most $\tau M$ missing cross neighbors. 
    Define $$B=|U||V|+e\left(H[U]\right)+e\left(H[V]\right)-\text{ex}(h,F_{2k}).$$ 
    Using $\text{ex}(n,F_{2k})+1=xy-M+A$ and (\ref{tunu}), we obtain
    $$B=M+1-O_k\left(\sqrt{M}+1\right)+\frac{k-2}{2}r+O_k(1)=M-O_k\left(\sqrt{M}+1\right).$$
    Hence, for sufficiently large $M$,
    $$B\geq \frac{3}{4}M.$$

We next show that a large value of $B$ forces many $W_{2k-1}$-saturating edges. \begin{claim}\label{hsat}
    There exist constants $C_k, c_k>0$, depending only on $k$, such that, if
$$|U|,|V|\geq \frac{h}{3},\ \ \tau M\leq\frac{h}{12(k+1)},\ \ B\geq C_k(\tau M+1),$$
then
\begin{align}
    f_{W_{2k+1}}(H)\geq c_kBh.
\end{align}
\end{claim}

\begin{proof}
     First observe that $$\Delta(H[U]),\Delta(H[V])\leq k-1.$$
    Otherwise, a vertex together with $k$ of its same-side neighbors would have at least $$\frac{h}{3}-(k+1) \tau M\geq \frac{h}{4}>k$$
    common neighbors on the opposite side. This yields a copy of $W_{2k+1}$, a contradiction.
    
    Let $$\alpha:=\frac{2e(H[U])}{|U|},\ \ \beta:=\frac{2e(H[V])}{|V|},\ \ |U|=\frac{h}{2}+z,\ \ |V|=\frac{h}{2}-z.$$
    Since $0\leq \alpha, \beta\leq k-1$, we have $\alpha-\beta\leq k-1$ and
    $$B=\frac{h}{4}(\alpha+\beta-(k-2))+\frac{(\alpha-\beta)^2}{16}-\left(z-\frac{\alpha-\beta}{4}\right)^2+O_k(1)\leq \frac{h}{4}(\alpha+\beta-(k-2))+O_k(1)\leq \frac{k}{4}h+O_k(1).$$
   Moreover, from the inequality above, since $B$ is sufficiently large, we have 
   \begin{align}\label{ab}
       \alpha+\beta\geq k-2+\frac{3B}{h}.
   \end{align}
   
   Write
   $$\theta:=\text{min}\left\{\frac{B}{1000kh}, \frac{1}{1000k}\right\}.$$ Then, for sufficiently large $h$,$$\theta|U|,\theta|V|\geq\frac{B}{3000k^2}.$$
   Suppose one side contains at least $$\frac{h}{12}\geq t\geq \frac{B}{2000k^2}$$ vertices such that every same-side non-edge incident with any of them is a $W_{2k+1}$-saturating edge. 
   Since $\Delta(H[U]),\Delta(H[V])\leq k-1$, each of them has at least $\frac{h}{3}-k$ same-side nonneighbors. Hence, the number of distinct $W_{2k+1}$-saturating edges incident with these $t$ vertices is at least
   $$\left(\frac{h}{3}-k\right)t-\binom{t}{2}.$$
   Therefore,
   $$f_{W_{2k+1}}(H)\geq \left(\frac{h}{3}-k\right)t-\binom{t}{2}\geq \left(\frac{h}{3}-k\right)t-\frac{t^2}{2}\geq t\left(\frac{7}{24}h-k\right)\geq \frac{ht}{4}\geq \frac{Bh}{8000k^2},$$
   where the penultimate inequality holds for sufficiently large $h$. Thus the claim follows with $$c_k=\frac{1}{8000k^2}.$$

   It remains to find at least $\frac{B}{2000k^2}$ such vertices. If $\beta\geq k-2+2\theta$, because $\Delta(H[V])\leq k-1$, then at least $2\theta|V|$ vertices have degree $k-1$. Now let $w\in V$ satisfy $d_{H[V]}(w)=k-1$, then every same-side non-edge $ww'$ is a $W_{2k+1}$-saturating edge. Indeed, in $H+ww'$, the $k$ same-side neighbors of $w$ together with $w$ have at least $$|U|-(k+1)\tau M\geq k$$ common neighbors in $U$, which form a $C_{2k}$ in $N_{H+ww'}(w)$.
   Since  $2\theta|V|\geq \frac{B}{2000k^2}$,
   $$f_{W_{2k+1}}(H)\geq c_kBh.$$

   We may therefore assume $\beta<k-2+2\theta$. By (\ref{ab}) and $0\leq \alpha, \beta\leq k-1$, the interval
   $$[k-2+2\theta-\alpha,\beta +1-2\theta)$$
   has length greater than 1, its left endpoint is greater than -1, and its right endpoint is less than $k-1$. Hence, it contains an integer $$s\in \{0, 1,\ldots,k-2\}.$$ Then $$\alpha\geq k-2-s+2\theta, \ \ \beta>s-1+2\theta.$$

   We next show that, after deleting any set of fewer than $(k+1)\tau M$ vertices from $V$, there still exist two disjoint copies of $P_{s+1}$ in $H[V]$.
   Otherwise, Lemma \ref{path} gives
   $$e(H[V])\leq \frac{s-1}{2}|V|+\Delta(H[V])\left((k+1)\tau M+s+1\right)=\frac{s-1}{2}|V|+O_k(\tau M+1),$$
   and hence
   $$\beta\leq s-1+O_k\left(\frac{\tau M+1}{|V|}\right).$$
   Since $B\geq C_k(\tau M+1)$, and $|V|\geq\frac{h}{3}$, choosing $C_k$ sufficiently large contradicts $\beta>s-1+2\theta$. When $s=0$, this simply means that at least two vertices remain.

Consider the two vertex-disjoint copies of $P_{s+1}$ obtained above, together with another $k-2-s$ vertices in $V$.
Applying Lemma \ref{use} yields $\Delta(H[U])\leq k-s-1$. Moreover, every same-side non-edge incident with a vertex of degree $k-s-1$ in $H[U]$ is a $W_{2k+1}$-saturating edge. 

   Since $\alpha\geq k-2-s+2\theta$, at least $$2\theta |U|\geq \frac{B}{2000k^2}$$ vertices of $U$ must have degree exactly $k-s-1$. Then $$f_{W_{2k+1}}(H)\geq c_kBh.$$
   This completes the proof of claim.
\end{proof}

Finally, choose $\tau>0$ sufficiently small that
$$C_k\tau \leq \frac{1}{4},\ \ \tau \leq\frac{1}{12k(k+1)}.$$
It is easy to know that, for sufficiently large $n$, $$|V|, |U|\geq \frac{h}{3}.$$
Combining (\ref{M}) with $h=n-O_k(1)$, for sufficiently large $n$, we have
$$\tau M\leq\frac{h}{12(k+1)}.$$
Moreover, $$B\geq \frac{3}{4}M\geq C_k(\tau M+1).$$
By Claim \ref{hsat}, $$f_{W_{2k+1}}(G)\geq f_{W_{2k+1}}(H)\geq c_kBh\geq \frac{1}{8}c_kMn,$$
contradicting the linear bound (\ref{upp}) once $M$ exceeds a sufficiently large constant depending only on $k$. Hence $$M=O_k(1).$$   
\end{proof}

We next prove $M=0$. We begin by bounding the maximum degrees in the two parts of $G$.

\begin{claim}\label{4}
    There exist integers $r\geq0$, $s\geq0$ with $$r+s=k-2$$ such that $$\Delta(G[X])\leq r+1,\ \ \Delta(G[Y])\leq s+1.$$
    Moreover, every same-side non-edge incident with a $(r+1)$-degree vertex of $X$, or $(s+1)$-degree vertex of $Y$, is a $W_{2k+1}$-saturating edge.
\end{claim}

\begin{proof}
    Since $M=O_k(1)$, $$\Delta(G[X]), \ \Delta(G[Y])\leq k-1,$$
    otherwise, a copy of $W_{2k+1}$ will be formed.
Choose the smallest integer $s\in\{0,1,\ldots, k-2\}$ such that $$e(G[Y])\leq \frac{s}{2}|Y|+C,$$
where $C=C(k)$ is sufficiently large, and put $$r:=k-2-s.$$
By the minimality of $s$, if $s\geq 1$, then $$e(G[Y])>\frac{s-1}{2}|Y|+C.$$
Hence, by Lemma \ref{path}, after deleting any $O_k(1)$ vertices from $Y$, the remaining graph still contains two vertex-disjoint copies of $P_{s+1}$. For $s=0$, this simply means that two vertices remain.

We next obtain the corresponding lower bound on $e(G[X])$. Since $r+s=k-2$,
\begin{align}\label{c}
    |X||Y|+\left\lfloor\frac{r|X|+s|Y|}{2}\right\rfloor\leq \text{ex}(n,F_{2k}).
\end{align}
Using
\begin{align}\label{e}
    e(G)=|X||Y|-M+e\left(G[X]\right)+e\left(G[Y]\right)=\text{ex}(n,F_{2k})+1,
\end{align}
 and $M=O_k(1)$, we get
\begin{align*}
    e(G[X])\geq \left\lfloor\frac{r|X|+s|Y|}{2}\right\rfloor+1-e(G[Y])\geq \frac{r}{2}|X|-O_k(1).
\end{align*}
    Therefore, by Lemma \ref{path} again, after deleting any $O_k(1)$ vertices from $X$, the remaining graph still contains two vertex-disjoint copies of $P_{r+1}$.

     We now consider these two copies together with another $r=k-2-s$ isolated vertices in $Y$. By Lemma \ref{use}, applied with $\zeta=r+2$ and $\gamma=M$, yields $$\Delta(G[X])\leq r+1,$$
     and every same-side non-edge with a $(r+1)$-degree vertex of $X$ is a $W_{2k+1}$-saturating edge.
     Similarly, $$\Delta(G[Y])\leq s+1,$$
     and every same-side non-edge with a $(s+1)$-degree vertex of $Y$ is a $W_{2k+1}$-saturating edge.
\end{proof}

\begin{claim}\label{5}
    $M=0$.
\end{claim}

\begin{proof}
    Let $r,s$ be given by Claim \ref{4}, and
    let $$q:=\left|{v\in X:d_{G[X]}(v)=r+1}\right|+\left|{v\in Y:d_{G[Y]}(v)=s+1}\right|.$$
    Also define
    $$\sigma:=\sum_{\substack{v\in X,\\ d_{G[X]}(v)\leq r}}(r-d_{G[X]}(v))+\sum_{\substack{v\in Y,\\ d_{G[Y]}(v)\leq s}}(s-d_{G[Y]}(v)).$$
    Then we obtain
    \begin{align}\label{aa}
        2A=2(e(G[X])+e(G[Y]))=r|X|+s|Y|+q-\sigma.
    \end{align}
By (\ref{c}), $$c:=\text{ex}(n,F_{2k})-|X||Y|-\left\lfloor\frac{r|X|+s|Y|}{2}\right\rfloor\geq 0.$$
    It follows from (\ref{e}) that 
    \begin{align}\label{a}
        A=\left\lfloor\frac{r|X|+s|Y|}{2}\right\rfloor+c+M+1.
    \end{align}
    Let $p\in \{0,1\}$. Combining (\ref{aa}) and (\ref{a}), we have
    $$q-\sigma=2c+2M+2-p\geq 1.$$
    Suppose that $M+c\geq 1$, then $q\geq3$.
    By Claim \ref{4}, $$f_{W_{2k+1}}(G)\geq \left(\frac{3}{2}+o(1)\right)n,$$
    contradicting (\ref{upp}). Thus $M+c=0$, in particular, $M=0$.
\end{proof}

\begin{claim}\label{struc}
    All internal edges of $G$ lie in one side of the maximum cut. More precisely, after possibly interchanging $X$ and $Y$,
    $$G[Y] \cong \overline{K}_y, \ \ \Delta(G[X])\leq k-1.$$
\end{claim}

\begin{proof}
It suffices to prove $e(G[Y])=0$.
  By Claim \ref{5}, $M+c=0$. Then $$q-\sigma\leq 2.$$  We first show that $r$ and $s$ cannot both be positive. Suppose, to the contrary, that $$r,s>0.$$
    As $q-\sigma\geq 1$, at least one vertex attains its upper internal degree. By symmetry, let $u\in X$ satisfy $$d_{G[X]}(u)=r+1.$$
    Since only $O_k(1)$ degree units in $Y$ differ from the baseline degree $s$,
    $$e(G[Y])=\frac{s}{2}|Y|+O_k(1).$$
Claim that $G[Y]$ contains no two vertex-disjoint copies of $P_{s+2}$. Otherwise, these two paths, together with $r-1$ additional isolated vertices, form a linear forest with $r+1$ components and $2(s+1)=2(k-(r+1))$ edges. By Lemma \ref{contain}, a copy of $W_{2k+1}$ is formed, a contradiction.

If $G[Y]$ contains a copy of $P_{s+2}$, delete its vertices; otherwise delete nothing. Let $Y'$ be the remaining set. Then $G[Y']$ is $P_{s+2}$-free, and 
\begin{align}\label{y'}
  e(G[Y'])=\frac{s}{2}|Y|-O_k(1).
\end{align}
 Let $L$ be the total number of vertices lying in components of $G[Y']$ containing no $P_{s+1}$. Then by Lemma \ref{expath},
$$e(G[Y'])\leq \frac{s}{2}(|Y'|-L)+\frac{s-1}{2}L=\frac{s}{2}|Y'|-\frac{L}{2}.$$
Together with (\ref{y'}),
$$L=O_k(1).$$
Thus  all but $O_k(1)$ vertices of $G[Y']$ lie in components containing a copy of $P_{s+1}$.
Moreover, as $G[Y']$ is $P_{s+2}$-free and $\Delta(G[Y'])\leq s+1\leq k-1$, every component of $G[Y']$ has order at most $$\sum_{i=0}^{s}(k-1)^i=O_k(1).$$ Hence $G[Y']$ has at least $$\frac{|Y'|}{O_k(1)}=\Omega_k(n)$$ components, and all but $O_k(1)$ vertices lie in components containing $P_{s+1}$. Therefore, there are $\Omega_k(n)$ components containing $P_{s+1}$. Call these components $good$.

Fix one good component $C_0$, and choose an edge $e_0$ in $C_0$. For every other good component $C_i$, choose a copy $P_i\cong P_{s+1}$, and let $u_i$ be an endpoint of $P_i$. For any two distinct good components $C_i,C_j\neq C_0$, the pair $u_iu_j$ is a non-edge of $G$. After adding $u_iu_j$, the two paths $P_i$ and $P_j$ combine into a copy of $P_{2s+2}$. Together with $e_0$ and $r-1$ additional isolated vertices, this gives a linear forest with $r+1$ components and $2s+2=2(k-(r+1))$ edges. Hence Lemma \ref{contain} shows that $u_iu_j$ is a $W_{2k+1}$-saturating edge.

Since there are $\Omega_k(n)$ good components, $$f_{W_{2k+1}}(G)=\Omega_k(n^2),$$ contradicting (\ref{upp}).
Thus $r$ and $s$ cannot both be positive. Without loss of generality, we may assume $$r=k-2,\ \ s=0.$$ 

By Claim \ref{4} and (\ref{upp}), $$\Delta(G[Y])\leq 1,\ \ e(G[Y])\leq 1.$$
It remains to rule out the $e(G[Y])=1$. Suppose $$E(G[Y])=\{uv\}.$$ Then $u$ and $v$ are both vertices of maximum internal degree. Hence there is no such vertex in $X$, and therefore
$$\Delta(G[X])\leq k-2.$$

Since $M=0$ and $c=0$, $|X|$ is a maximizing value in (\ref{1}). Then
\begin{align}\label{eX}
   |X||Y|+e(G[X])+1& =\text{ex}(n,F_{2k})+1=|X|(n-|X|)+\left\lfloor\frac{k-2}{2}|X|\right\rfloor+1 \notag \\
   e(G[X])&=\left\lfloor\frac{k-2}{2}|X|\right\rfloor.
\end{align}
For sufficiently large $n$, Lemma \ref{expath} and (\ref{eX}) imply that $G[X]$ contains two vertex-disjoint copies of $P_{k-1}$. Now take $w\in Y\backslash \{u,v\}$. Applying Lemma \ref{use} with $U=Y$, $V=X$, $\gamma=0$, and $\zeta=2$, every same-side non-edge incident with $u$ or $v$ is a $W_{2k+1}$-saturating edge.
Consequently, 
\begin{align}\label{6}
    f_{W_{2k+1}}(G)\geq 2(|Y|-2).
\end{align}
We finally compare (\ref{6}) with the sharp value in Theorem \ref{odd}. 
Put $$z:=4|X|-(2n+k-2).$$
Since $|X|$ is a maximizing value in (\ref{1}), its deviation from $\ell_k$ is bounded by the rounding error. Hence $$z\leq 3.$$ 

If $(k-2)|X|$ is odd, $2(|Y|-2)-(|X|-k)>0$ for sufficiently large $n$, contradicting extremality.
Suppose now that $(k-2)|X|$ is even, then $$2(|Y|-2)-(2(|X|-k)-1)=k-1-z.$$
If $k\geq 5$, then $|z|\leq 3$ gives $k-1-z>0$. If $k=4$, then $z$ is even, so $z\leq 2$, and again $k-1-z>0$. If $k=3$, the only remaining case is $z=3$. In this case, $|X|$ is even. Since $\left\lfloor\frac{|X|}{2}\right\rfloor=\left\lfloor\frac{|X|-1}{2}\right\rfloor$, $|A|-1$ is also a maximizing value of (\ref{1}). 
As $|X|-1$ is odd, then the sharp value in Theorem \ref{odd} is at most $|X|-4<2|Y|-4$ for sufficiently large $n$. Thus every cases contradicts extremality.
Therefore, $$e(G[Y])=0.$$
\end{proof}
We first note that $G[X]$ is $P_{2k-1}$-free. Indeed, otherwise every non-edge in the independent side is $W_{2k+1}$-saturating edge, giving $\Omega(n^2)$ $W_{2k+1}$-saturating edges, contrary to (\ref{upp}).

Let $q$ be the number of vertices of $G[X]$ of degree $k-1$. Then $$q\geq 
\begin{cases}
    1, & (k-2)|X| \text{ is odd},\\
    2, & (k-2)|X| \text{ is even}.
\end{cases}$$
By Lemma \ref{sat1}, Proposition \ref{exist} and the minimality of $G$,
\begin{align}\label{7}
    f_{W_{2k+1}}(G)= \begin{cases}
    |X|-k, & (k-2)|X| \text{ is odd},\\
    2(|X|-k)-1, & (k-2)|X| \text{ is even}.
\end{cases}
\end{align}
Moreover, $$q= 
\begin{cases}
    1, & (k-2)|X| \text{ is odd},\\
    2, & (k-2)|X| \text{ is even}.
\end{cases}$$

Set $$\sigma:=\sum_{\substack{v\in X,\\ d_{G[X]}(v)\leq k-2}}(k-2-d_{G[X]}(v)).$$
Since $G\cong G[X]\vee \overline{K}_y$, where $|X|$ is a maximizing of (\ref{1}),
\begin{align}\label{last}
   |X||Y|+e(G[X])& =\text{ex}(n,F_{2k})+1=|X|(n-|X|)+\left\lfloor\frac{k-2}{2}|X|\right\rfloor+1 \notag \\
   e(G[X])&=\left\lfloor\frac{k-2}{2}|X|\right\rfloor+1.
\end{align}
Then
$$q-\sigma=\begin{cases}
    1, & (k-2)|X| \text{ is odd},\\
    2, & (k-2)|X| \text{ is even},
\end{cases}$$
and $$\sigma=0.$$
Hence we determine the degree sequence of the equality case.
It remains only to show that the two vertices of degree $k-1$ are nonadjacent when $(k-2)|X|$ is even. If they are adjacent, their incident non-edge sets would be disjoint, and hence that would give at least $$2(\ell_k-k)$$ $W_{2k+1}$-saturating edges, contrary to (\ref{7}).
This completes the proof of Theorem \ref{odd}.

\section{Acknowledgments}
The authors would like to thank Professor Hong Liu for bringing the problem to our attention,
and also for  insightful  discussions on the problem.

\section{Declaration of AI use}
During the preparation of this work, the authors used OpenAI ChatGPT to assist in the development of certain mathematical proofs. All AI-assisted arguments were independently verified by authors, who take full responsibility for correctness of proofs and the content of manuscript.

	\end{document}